\documentclass[12pt]{amsart}
\pdfoutput=1
\usepackage{amsfonts}
\usepackage{amssymb}
\usepackage{amsmath}
\usepackage{amsrefs}
\usepackage{fullpage}
\usepackage{tikz-cd}
\usepackage{enumerate}
\usepackage{dutchcal}

\newtheorem{theorem}{Theorem}[section]
\newtheorem{lemma}[theorem]{Lemma}
\newtheorem{proposition}[theorem]{Proposition}
\newtheorem{corollary}[theorem]{Corollary}
\theoremstyle{definition}
\newtheorem{definition}[theorem]{Definition}
\theoremstyle{remark}
\newtheorem{remark}[theorem]{Remark}

\renewcommand{\:}{\colon}
\newcommand{\inv}{^{-1}}

\newcommand{\suchthat}{\,|\,}
\renewcommand{\And}{\mathbin{\wedge}}
\newcommand{\Or}{\mathbin{\vee}}
\newcommand{\Not}{\neg}
\newcommand{\IM}{\quad\Longrightarrow\quad}
\newcommand{\EV}{\quad\Longleftrightarrow\quad}
\newcommand{\cat}{\mathsf}
\newcommand{\FF}{\mathbb{F}}
\newcommand{\RR}{\mathbb{R}}
\newcommand{\CC}{\mathbb{C}}
\newcommand{\inc}{\mathrm{inc}}
\newcommand{\coker}{\mathrm{coker}}
\newcommand{\id}{\mathrm{id}}
\newcommand{\iso}{\cong}
\newcommand{\Tr}{\mathrm{Tr}}
\newcommand{\Fam}{\mathrm{Fam}}
\newcommand{\Mor}{\mathrm{Mor}}
\newcommand{\tensor}{\otimes}

\allowdisplaybreaks

\begin{document}

\title{Axioms for the category of sets and relations}
\author{Andre Kornell}
\address{Department of Mathematics and Statistics, Dalhousie University, Halifax, Nova Scotia}
\email{akornell@dal.ca}
\thanks{This work was
supported by the Air Force Office of Scientific Research under Award No.~FA9550-21-1-0041.}

\begin{abstract}
We provide axioms for the dagger category of sets and relations that recall recent axioms for the dagger category of Hilbert spaces and bounded operators.
\end{abstract}

\maketitle

\section{Introduction}\label{introduction}

A \emph{dagger category} is a category $\cat{C}$ with an operation $(-)^\dag\: \Mor(\cat{C}) \to \Mor(\cat{C})$ such that
\begin{enumerate}
\item $\id_X^\dag = \id_X$ for each object X;
\item $f^{\dag \dag} = f$ for each morphism $f$;
\item $(f \circ g)^\dag = g^\dag \circ f^\dag$ for all composable pairs $(f, g)$.
\end{enumerate}
Two prominent examples of dagger categories are $\cat{Rel}$, the dagger category of sets and binary relations, and $\cat{Hilb}_\FF$, the dagger category of Hilbert spaces and bounded operators over $\FF$, where $\FF = \RR$ or $\FF = \CC$. For a binary relation $r$, the binary relation $r^\dag$ is the converse of $r$, and for a bounded operator $a$, the bounded operator $a^\dag$ is the Hermitian adjoint of $a$.

The dagger categories $\cat{Rel}$ and $\cat{Hilb}_\FF$ have many properties in common. These properties may be expressed in terms of morphisms that behave like the embedding of one object into another. These morphisms are characterized by the conjunction of two familiar properties: a morphism $m\:X \to Y$ is said to be a dagger monomorphism if $m^\dag \circ m = \id_X$, and it is said to be a normal monomorphism if it is the kernel of some morphism $Y \to Z$.

Following Heunen and Jacobs, we use the term \emph{dagger kernel} for morphisms that are both dagger monomorphisms and normal monomorphisms \cite{HeunenJacobs}. Heunen and Jacobs showed that in any dagger category satisfying axioms A and B, below, each dagger kernel $m$ has a \emph{complement} $m^\perp$, which is a kernel of $m^\dag$. Explicitly, they showed that $m$ and $m^{\perp\perp}$ are isomorphic as morphisms into their shared codomain. Two dagger kernels $m$ and $n$ are said to be \emph{orthogonal} if $m^\dag \circ n$ is zero or, equivalently, if $m$ factors through $n^\perp$. A dagger kernel in $\cat{Rel}$ is an injective function, and a dagger kernel in $\cat{Hilb}_\FF$ is a linear isometry.

The dagger categories $\cat{Rel}$ and $\cat{Hilb}_\FF$ are also symmetric monoidal categories when they are equipped with the Cartesian product and the tensor product, respectively. They are said to be \emph{dagger symmetric monoidal categories} because their symmetric monoidal structures are compatible with their dagger structures: their monoidal products preserve the dagger operation in the obvious way, and their coherence isomorphisms are all dagger kernels. 

The dagger symmetric monoidal categories $\cat{Rel}$ and $\cat{Hilb}_\FF$ both satisfy the following axioms:

\begin{enumerate}[ (A)]
\item there is a zero object;
\item each morphism has a kernel that is a dagger kernel;
\item each pair of complementary dagger kernels is jointly epic;
\item each pair of objects has a coproduct whose inclusions are orthogonal dagger kernels;
\item the monoidal unit is not a zero object;
\item each nonzero endomorphism of the monoidal unit is invertible;
\item the monoidal unit is a monoidal separator.
\end{enumerate}
An object $I$ is said to be a separator in the case that the morphisms $a\:I \to X$ are jointly epic, for all objects $X$. It is said to be a \emph{monoidal separator} in the case that the morphisms $a \tensor b\: I \tensor I \to X \tensor Y$ are jointly epic, for all objects $X$ and $Y$. Axiom~G refers to this property. For further glosses of these axioms, see section~\ref{dagger categories}.

The shared axioms A--G are almost sufficient to axiomatize both
$\cat{Rel}$ and $\cat{Hilb}_\FF$:
\begin{theorem}\label{A}
Let $(\cat{C}, \tensor, I, \dag)$ be a dagger symmetric monoidal category that satisfies axioms A--G. Then,
\begin{enumerate}[ (i)]
\item $(\cat{C}, \tensor, I, \dag)$ is equivalent to $(\cat{Rel}, \times, \{\ast\}, \dag)$ if and only if every object has a dagger dual and every family of objects has a coproduct whose inclusions are pairwise-orthogonal dagger kernels;
\item $(\cat{C}, \tensor, I, \dag)$ is equivalent to $(\cat{Hilb_\FF}, \tensor, \FF^{1}, \dag)$ for $\FF = \RR$ or $\FF = \CC$ if and only if every dagger monomorphism is a dagger kernel and the wide subcategory of dagger kernels has directed colimits.
\end{enumerate}
\end{theorem}

\noindent This pair of equivalences provides a category-theoretic perspective on the analogy that is sometimes drawn between sets and Hilbert spaces \cite{Weaver}.

Dagger categories have been considered for more than half of a century \cite{BrinkmannPuppe}*{Definition~6.4.1}. Interest in dagger categories in the context of categorical quantum information theory began with \cite{AbramskyCoecke}. The term originates in \cite{Selinger}. The axiomatizations of $\cat{Hilb}_\RR$ and $\cat{Hilb}_\CC$ in \cite{HeunenKornell} derive from Sol\`{e}r's theorem \cite{Soler}. Axiomatizations of $\cat{Con}_\RR$ and $\cat{Con}_\CC$, the categories of Hilbert spaces and contractions, have also been obtained \cite{HeunenKornellVanDerSchaaf}.

The classic work of Lawvere provides axioms for the category $\cat{Set}$ of sets and functions \cite{Lawvere}. The close relationship between $\cat{Set}$ and $\cat{Rel}$ and the similarity between Lawvere's assumption of limits and our assumption of biproducts naturally invite a comparison between \cite{Lawvere}*{Corollary} and Theorem~\ref{V}. Unlike Lawvere, we have not chosen our axioms to provide a foundation for mathematics but rather to draw a comparison between the category $\cat{Rel}$ and the categories $\cat{Hilb}_\FF$, as in Theorem~\ref{A}. Less directly, our assumptions about dagger kernels derive from \cite{Selinger2}, \cite{Vicary}, and \cite{HeunenJacobs}, and even less directly, they derive from elementary results on abelian categories \cite{MacLane}. Nevertheless, we refer the reader to Corollary~\ref{W}.

Lawvere's axiomatization of $\cat{Set}$ can be transformed into an axiomatization of $\cat{Rel}$ as an allegory \cite{FreydScedrov}*{2.132}. An allegory is a dagger category that is enriched over posets with meets and that satisfies the \emph{law of modularity}: $t \wedge (s \circ r) \leq s \circ ((s^\dag \circ t) \wedge r)$ \cite{FreydScedrov}*{2.11}. The resulting axiomatization of $\cat{A} = \cat{Rel}$ asserts that $\mathcal{Map}(\cat{A})$ satisfies Lawvere's axioms and that $\mathcal{Rel}(\mathcal{Map}(\cat{A})) = \cat{A}$ in the sense that $(f,g) \mapsto g \circ f^\dag$ defines an equivalence of categories \cite{FreydScedrov}*{1.56}. The novelty of Theorem~\ref{A}(i) relative to this older axiomatization of $\cat{Rel}$ is that every axiom except one is also satisfied by $\cat{FinHilb}_\FF$ and that enrichment over posets is proved rather than assumed.

\section{Dagger symmetric monoidal categories}\label{dagger categories}

This section explains the terms in Theorem~\ref{A}; the reader may also wish to consult \cite{HeunenVicary}. Throughout, we illustrate the terms using $\cat{Rel}$ and $\cat{Hilb}_\FF$, where $\FF = \RR$ or $\FF = \CC$. In $\cat{Rel}$, an object is a set, and a morphism $r\:X \to Y$ is a subset $r \subseteq X \times Y$. We say that $r$ is a relation from $X$ to $Y$, and when $(x,y) \in r$, we say that $r$ relates $x$ to $y$. For $r\:X \to Y$ and $s\: Y \to Z$, the composition $s \circ r$ relates $x$ to $z$ if there exists $y \in Y$ such that $r$ relates $x$ to $y$ and $s$ relates $y$ to $z$. In $\cat{Hilb}_\FF$, an object is a Hilbert space over $\FF$, and a morphism $r\: X \to Y$ is a bounded linear operator.

A \emph{dagger category} is commonly defined to be a category $\cat{C}$ with a contravariant functor $(-)^\dag$ that is identity on objects and an involution on morphisms. However, this definition is incompatible with the principle of equivalence \cite{Karvonen}*{section~3.1}, and for this reason, the definition given in section~\ref{introduction} may be preferred. Thus, dagger categories are viewed as a variant notion of categories, rather than as categories that are equipped with additional structure. In $\cat{Rel}$, the dagger of $r\: X \to Y$ is the converse relation $r^\dag$ that relates $y$ to $x$ if $r$ relates $x$ to $y$. In $\cat{Hilb}_\FF$, the dagger of $r\:X \to Y$ is the Hermitian adjoint operator $r^\dag$, which is defined by $\langle r^\dag y | x \rangle = \langle y | rx \rangle$ for all $x \in X$ and $y \in Y$.

A number of basic concepts for categories have canonical analogs for dagger categories \cites{AbramskyCoecke2,Karvonen}. Prominently, a \emph{dagger isomorphism} is a morphism $u\: X \to Y$ such that $u^\dag \circ u = \id_X$ and $u \circ u^\dag = \id_Y$. In $\cat{Rel}$, the dagger isomorphisms are the bijections, and in $\cat{Hilb}_\FF$, the dagger isomorphisms are the unitary operators. Dagger isomorphisms replace isomorphisms in a variety of familiar contexts. For example, a \emph{dagger equivalence} consists of dagger functors $F \: \cat{C} \to \cat{D}$ and $G\: \cat{D} \to \cat{C}$ with natural dagger isomorphisms $G \circ F \iso \id_{\cat{C}}$ and $F \circ G \iso \id_{\cat{D}}$. A \emph{dagger functor} is, of course, a functor $F$ such that $F(r^\dag) = F(r)^\dag$ for all morphisms $r$.

A \emph{dagger symmetric monoidal category} is a dagger category $\cat{C}$ that is equipped with symmetric monoidal structure that is compatible with the dagger operation in two ways: first, $(r \tensor s)^\dag = r^\dag \tensor s^\dag$ for all morphisms $r$ and $s$, and second, the associators, braidings, and unitors are all dagger isomorphisms. Both $\cat{Rel}$ and $\cat{Hilb}_\FF$ are canonically dagger symmetric monoidal categories. In $\cat{Rel}$, the product $X \tensor Y$ is the Cartesian product of $X$ and $Y$, and $r \tensor s$ relates $(x_1, y_1)$ to $(x_2, y_2)$ if $r$ relates $x_1$ to $x_2$ and $s$ relates $y_1$ to $y_2$. The monoidal unit is a chosen singleton $\{\ast\}$. In $\cat{Hilb}_\FF$, the product $X \tensor Y$ is the tensor product of $X$ and $Y$, and $r \tensor s$ maps $x \tensor y$ to $rx \tensor sy$. The monoidal unit is the Hilbert space $\FF^1$.

The dagger symmetric monoidal categories $\cat{Rel}$ and $\cat{Hilb}_\FF$ both satisfy the following axioms:
\begin{enumerate}[ (A)]
\item There is a \emph{zero object}: there exists an object that is both terminal and initial. In $\cat{Rel}$, the zero object is the empty set. In $\cat{Hilb}_\FF$, a zero object is a zero-dimensional Hilbert space, e.g., $\FF^0$. A \emph{zero morphism} is a morphism that factors through a zero object. For all objects $X$ and $Y$, there is a unique zero morphism $0_{X, Y}\: X \to Y$. In $\cat{Rel}$, a zero morphism is an empty relation. In $\cat{Hilb}_\FF$, a zero morphism is a zero operator.
\item Each morphism has a kernel that is a dagger kernel: for each morphism $r\: X \to Y$, there is a morphism $m\: A \to X$ such that $m^\dag \circ m = \id_A$, such that $r \circ m = 0_{A,Y}$, and such that each morphism $s\: Z \to X$ factors uniquely through $m$ if $r \circ s = 0_{Z,Y}$.
$$
\begin{tikzcd}
A
\arrow{r}{m}
&
X
\arrow{r}{r}
&
Y
\\
Z
\arrow{ur}{s}
\arrow{urr}[swap]{0_{Z,Y}}
\arrow[dotted]{u}{!}
&
&
\end{tikzcd}
$$
In this case, $m$ is said to be a \emph{dagger kernel} of $r$ or just a \emph{dagger kernel}. In $\cat{Rel}$, the dagger kernels are exactly the injections, and for each relation $r\: X \to Y$, the inclusion function of the subset $A = \{x \in X \suchthat (x, y) \not \in r \text{ for all } y \in Y\}$ is a dagger kernel of $r$. In $\cat{Hilb}_\FF$, the dagger kernels are exactly the isometries, and for each operator $r\: X \to Y$, the incusion operator of the subspace $A = \{x \in X \suchthat rx = 0\}$ is a dagger kernel of $r$.

A \emph{complement} of a dagger kernel $m$ is a dagger kernel of $m^\dag$. By \cite{HeunenJacobs}*{Lemma~1}, if $n$ is a complement of $m$, then $m$ is a complement of $n$. In this case, we say that $m$ and $n$ are \emph{complementary} dagger kernels. In $\cat{Rel}$, two injections $m\: A \to X$ and $n\: B \to X$ are complementary iff the range of $m$ and the range of $n$ are complements as subsets of $X$. In $\cat{Hilb}_\FF$, two isometries $m\: A \to X$ and $n\: B \to X$ are complementary iff the range of $m$ and the range of $n$ are orthogonal complements as subspaces of $X$.
\item Each pair of complementary dagger kernels is jointly epic: for all dagger kernels $m\: A \to X$ and $n\: B \to X$ and all morphism $r_1, r_2\: X \to Y$, if $r_1 \circ m = r_2 \circ m$, $r_1 \circ n = r_2 \circ n$, and $n$ is a complement of $m$, then $r_1 = r_2$. In $\cat{Rel}$, each pair of complementary injections, $m$ and $n$, is jointly epic because the union of their ranges is $X$. In $\cat{Hilb}_\FF$, each pair of complementary isometries, $m$ and $n$, is jointly epic because the union of their ranges spans $X$.
\item Each pair of objects has a coproduct whose inclusions are orthogonal dagger kernels: for all objects $X$ and $Y$, there exist an object $W$ and dagger kernels $i\: X \to W$ and $j\: Y \to W$ such that $j^\dag \circ i = 0_{X,Y}$ and such that each pair of morphisms $r\: X \to Z$ and $s\: Y \to Z$ factors uniquely through $i$ and $j$.
$$
\begin{tikzcd}
X
\arrow{r}{i}
\arrow{rd}[swap]{r}
&
W
\arrow[dotted]{d}{!}
&
Y
\arrow{l}[swap]{j}
\arrow{ld}{s}
\\
&
Z
&
\end{tikzcd}
$$
In this case the object $W$, together with the morphisms $i$ and $j$, is said to be a \emph{dagger biproduct} of $X$ and $Y$. It is both their product and their coproduct. In $\cat{Rel}$, a dagger biproduct of sets $X$ and $Y$ is a disjoint union of $X$ and $Y$, and in $\cat{Hilb}_\FF$, a dagger biproduct of Hilbert spaces $X$ and $Y$ is a direct sum of $X$ and $Y$.
\item The monoidal unit is not a zero object: $I$ is neither initial nor terminal. In $\cat{Rel}$, we have that $\{\ast\} \not \iso \emptyset$ because cardinality is an isomorphism invariant. In $\cat{Hilb}_\FF$, we have that $\FF^1 \not \iso \FF^0$ because dimension is an isomorphism invariant.
\item Each nonzero endomorphism of the monoidal unit is invertible: every nonzero morphism $a\: I \to I$ has an inverse. In $\cat{Rel}$, there is exactly one nonzero endomorphism of the monoidal unit $I = \{\ast\}$. In $\cat{Hilb}_\FF$, the nonzero endomorphisms of the monoidal unit $I = \FF^1$ are exactly the nonzero elements of $\FF$, which form a group because $\FF$ is a field.
\item The monoidal unit is a monoidal separator: for all distinct $r_1, r_2\: X \tensor Y \to Z$, there exist $a\: I \to X$ and $b\: I \to Y$ such that $f_1 \circ (a \tensor b) \neq f_2\circ (a \tensor b)$. In $\cat{Rel}$, morphisms $\{\ast\} \to X$ for a set $X$ correspond exactly to the elements of $X$, so the monoidal unit is a monoidal separator because the Cartesian product $X \tensor Y$ consists of pairs of elements. In $\cat{Hilb}_\FF$, morphisms $\FF^1 \to X$ correspond exactly to the vectors in $X$, so the monoidal unit is a monoidal separator because elementary tensors span the tensor product $X \tensor Y$.
\end{enumerate}

Axioms A--G prescribe the existence of various objects and morphisms that need not be unique. For convenience, we introduce notations for specific choices of these objects and morphisms. In $\cat{Rel}$, we write $0 = \emptyset$, and in $\cat{Hilb}_\FF$, we write $0 = \FF^0$. In $\cat{Rel}$, we write $X \oplus Y$ for the standard disjoint union of sets $X$ and $Y$, and in $\cat{Hilb}_\FF$, we write $X \oplus Y$ for the standard direct sum of Hilbert spaces $X$ and $Y$. In $\cat{Rel}$, we write $\ker(r)$ for the inclusion function that is a dagger kernel of a relation $r\: X \to Y$, and in $\cat{Hilb}_\FF$, we write $\ker(r)$ for the inclusion operator that is a dagger kernel of an operator $r\: X \to Y$.

In an arbitrary dagger symmetric monoidal category that satisfies axioms A--G, it may not be possible to choose a dagger kernel $m = \ker(r)$ for each morphism $r\: X \to Y$, and it may not be possible to choose a dagger biproduct $W = X \oplus Y$ for each pair of objects $X$ and $Y$, because the objects of the category may form a proper class. In this case, we make such choices only as necessary. We can always choose dagger kernels so that $\ker(r_1) = \ker(r_2)$ whenever $\ker(r_1)$ and $\ker(r_2)$ represent the same subobject of $X$, i.e., whenever $r_2 = r_1 \circ i$ for some isomorphism $i$. Our choice of canonical dagger kernels in $\cat{Rel}$ and in $\cat{Hilb}_\FF$ follows this convention, and it significantly reduces clutter.

The dagger symmetric monoidal category $\cat{FinHilb}_\FF$ of finite-dimensional Hilbert satisfies axioms A--G as well, and in some respects, it occupies a middle ground between $\cat{Rel}$ and $\cat{Hilb}_\FF$. For example, in both $\cat{FinHilb}_\FF$ and $\cat{Hilb}_\FF$, every dagger monomorphism is a dagger kernel, i.e., every morphism $m\: A \to X$ that satisfies $m^\dag \circ m = \id_A$ is the kernel of some morphism $r\:X \to Y$. This does not occur in $\cat{Rel}$. For example, the relation $f^\dag: A \to X$ is a dagger monomorphism but not a dagger kernel when $A$ is a singleton, $X$ is a pair, and $f\: X \to A$ is a function. On the other hand, in both $\cat{FinHilb}_\FF$ and $\cat{Rel}$, every object has a dagger dual, but this does not occur in $\cat{Hilb}_\FF$. We review dagger duals now.

Let $\alpha$, $\beta$, and $\gamma$, denote the associator, braiding, and left unitor of a dagger symmetric monoidal category, respectively. A \emph{dagger dual} of an object $X$ is an object $X^*$ together with a morphism $\eta_X\: I \to X^* \tensor X$ such that
$$
\bar \gamma_X \circ (\id_X \tensor \eta_X^\dagger) \circ \alpha_{X, X^*, X} \circ (\eta_{X^*} \tensor \id_X) \circ \gamma_X^\dagger = \id_X
$$
$$
\bar \gamma_{X^*} \circ (\id_{X^*} \tensor \eta_{X^*}^\dagger) \circ \alpha_{X^*, X, X^*} \circ (\eta_X \tensor \id_{X^*}) \circ \gamma_{X^*}^\dagger = \id_{X^*},
$$
where $\bar \gamma_X = \gamma _X \circ \beta_{X, I}$, $\bar \gamma_{X^*} = \gamma_{X^*} \circ \beta_{X^*, I}$, and $\eta_{X^*} = \beta_{X^*,X} \circ \eta_X$. In this case, the object $X$ together with the morphism $\eta_{X^*}\: I \to X \tensor X^*$ is a dagger dual of $X^*$ as well. Of course, a dagger dual of $X$ is also a dual of $X$ in the standard sense \cite{MacLane2}. A dagger symmetric monoidal category in which every object has a dagger dual has been called a strongly compact closed category \cite{AbramskyCoecke2} and then a \emph{dagger compact closed category} \cite{Selinger}.

In $\cat{Rel}$, the dagger dual of a set $X$ is the same set $X^* = X$ together with the relation $\eta_X$ that relates the unique element of the monoidal unit to all pairs of the form $(x,x)$ for $x \in X$. In $\cat{FinHilb}_\FF$, the dagger dual of a finite-dimensional Hilbert space $X$ is the conjugate Hilbert space $X^* = \overline X$ together with the operator $\eta_X\: 1 \mapsto \sum_{i\in M} \overline{e_{i}} \tensor e_i$, where $\{e_i \suchthat i \in N\}$ is any orthonormal basis of $X$ and $\{\overline{e_{i}} \suchthat i \in N\}$ is the corresponding orthonormal basis of $\overline X$. In an arbitrary dagger compact closed category, it may not be possible to choose a dagger dual $X^*$ for each object $X$ because the objects of the category may form a proper class.
In this case, we make such choices only as necessary. We can always choose dagger duals so that $X^{**} = X$. Our choice of canonical dagger duals in $\cat{Rel}$ and in $\cat{FinHilb}_\FF$ follows this convention. In $\cat{Hilb}_\FF$, no infinite-dimensional Hilbert space has a dagger dual \cite{HeunenVicary}*{example~3.2}.

\section{Infinite biproducts and complete semirings}\label{complete semirings}

The biproduct $\oplus$ is classically defined in the setting of abelian categories \cite{MacLane2}. In any abelian category, we have that $f + g = \nabla_Y \circ (f \oplus g) \circ \Delta_X$, where $\Delta_X\: X \to X \oplus X$ and $\nabla_Y\: Y \oplus Y \to Y$ are the diagonal and the codiagonal morphisms, respectively. This equation provides a bridge to an alternative definition of abelian categories, in which no enrichment is assumed \cites{Freyd, Schubert}. In this context, a \emph{biproduct} of objects $X$ and $Y$ is an object $X \oplus Y$ together with ``projections'' $p\: X \oplus Y \to X$ and $q\: X \oplus Y \to Y$ and ``inclusions'' $i\: X \to X \oplus Y$ and $j\: Y \to X \oplus Y$ such that $(X \oplus Y, p, q)$ is a product, such that $(X \oplus Y, i, j)$ is a coproduct, and such that $p \circ i = \id_X$, $q \circ j = \id_Y$, $q \circ i = 0_{X,Y}$, and $p \circ j = 0_{Y,X}$.

Neither $\cat{Rel}$ nor $\cat{Hilb_\FF}$ are abelian categories. Fortunately, biproducts yield a canonical enrichment over commutative monoids in a more general setting that includes both of these categories \cite{MacLane}*{section 19}. In $\cat{Rel}$, each infinite family of objects has a biproduct, and this property distinguishes $\cat{Rel}$ from $\cat{Hilb}_\FF$. This means that for any family of objects $\{X_\alpha\}_{\alpha \in M}$, there exists an object $X = \bigoplus_{\alpha \in M} X_\alpha$ together with ``projections'' $p_\alpha \: X \to X_\alpha$ that make $X$ a product and ``inclusions'' $i_\alpha\: X_\alpha \to X$ that make $X$ a coproduct such that $p_\alpha \circ i_\alpha = \id_{X_\alpha}$ and otherwise $p_\alpha \circ i_{\beta} = 0_{X_\beta,X_\alpha}$. These biproducts yield a canonical enrichment over complete monoids in a straightforward generalization of the finite case; see \cite{Laird}*{Proposition~2.3} and \cite{Haghverdi}*{Theorem 3.0.17}.

A complete monoid is an abelian monoid in which one can form the sum of any indexed family of elements. For each set $R$, let $\Fam(R)$ be the class of all indexed families of elements of $R$. Formally, a \emph{complete monoid} is a set $R$ together with an operation $\Sigma\: \Fam(R) \to R$ that maps singleton families to their elements and that satisfies the associativity condition
$$
\sum_{\alpha \in M} r_\alpha = \sum_{\beta \in N} \sum_{\alpha \in f^{-1}(\beta)} r_\alpha
$$
for every function $f\: M \to N$ \cites{Haghverdi2,Laird}. If $\{r_\alpha\}_{\alpha \in M}$ is a family of morphisms $X \to Y$ in a category with biproducts for all indexed families of objects, then
$$\sum_{\alpha \in M} r_\alpha := \nabla \circ \left( \bigoplus_{\alpha \in M} r_\alpha \right) \circ \Delta,$$
where $\Delta\: X \to \bigoplus_{\alpha \in M} X$ is the diagonal map and $\nabla\: \bigoplus_{\alpha \in M} Y \to Y$ is the codiagonal map.

Thus, for each object $X$ in a category with all biproducts, the set of all morphisms $X \to X$ is both a complete monoid with respect to the operation $\sum$ and a monoid with respect to the operation $\circ$. Enrichment over complete monoids implies that the latter operation distributes over the former operation in the sense that 
$$
\sum_{\alpha \in M} s \circ r_\alpha = s \circ \left( \sum_{\alpha \in M} r_\alpha \right),
\qquad \qquad
\sum_{\alpha \in M} r_\alpha \circ s = \left( \sum_{\alpha \in M} r_\alpha \right) \circ s.
$$
In other words, the endomorphisms of $X$ form a \emph{complete semiring} \cites{Golan, Laird}.

In the setting of dagger categories, a \emph{dagger biproduct} of an indexed family $\{X_\alpha\}_{\alpha \in M}$ is a biproduct $\bigoplus_{\alpha \in M} X_\alpha$ such that $p_\alpha = i_\alpha^\dag$ for each $\alpha \in M$. It follows that the morphisms $i_\alpha$ are dagger monomorphisms in the sense that $i_\alpha^\dag \circ i_\alpha = \id_{X_\alpha}$ and that they are pairwise-orthogonal in the sense that $i_\alpha^\dag \circ i_\beta = 0_{X_\beta, X_\alpha}$ for $\alpha \neq \beta$. Thus, the existence of dagger biproducts for all indexed families of objects implies axiom D. In the case of dagger biproducts, the diagonal map $\Delta\: X \to \bigoplus_{\alpha \in M} X$ and the codiagonal map $\nabla\: \bigoplus_{\alpha \in M} X \to X$ are related by $\nabla = \Delta^\dag$. It follows that in a dagger category with all dagger biproducts, the endomorphisms of $X$ form a complete semiring with sums that are defined by
$$
\sum_{\alpha \in M} r_\alpha := \Delta^\dag \circ \left( \bigoplus_{\alpha \in M} r_\alpha \right) \circ \Delta.
$$
The operation $\dag$ is an involution that satisfies $(r \circ s)^\dag = s^\dag \circ r^\dag$ and $\left( \sum_{\alpha \in M} r_\alpha \right)^\dag = \sum_{\alpha \in M} r_\alpha^\dag$.

\section{Complete Boolean algebras}\label{part1}

Let $(\cat{C}, \tensor, I, \dag)$ be a dagger symmetric monoidal category with dagger biproducts for all families of objects. Assume that every morphism has a kernel that is dagger monic and that $k$ and $k^\perp: = \ker(k^\dagger)$ are \emph{jointly epic} for every dagger kernel $k$. The latter condition means that $f = g$ whenever $f \circ k = g \circ k$ and $f \circ k^\perp = g \circ k^\perp$. Further, assume that $I$ is a separator, that $I$ is nonzero, and that all nonzero morphisms $I \to I$ are invertible.
In this section, we show that for each object $X$, morphisms $I \to X$ form a complete Boolean algebra. First, we use an Eilenberg swindle to show that the scalars of $\cat{C}$ must be the Boolean algebra $\{0,1\}$.

\begin{lemma}\label{C}
Let $(R, \Sigma, \cdot)$ be a complete semiring, and let $R^\times = R \setminus \{0\}$. If $(R^\times, \,\cdot\,)$ is a group, then $R^\times = \{1\}$, and $1 + 1 = 1$.
\end{lemma}

\begin{proof}
Let $\omega = 1 + 1 + \cdots$. Clearly $\omega + \omega = \omega$. Furthermore, $\omega \neq 0$, because equality would imply that $0 = \omega = \omega + 1 = 0 + 1 = 1$. We now calculate that $1 + 1 = \omega\inv \cdot \omega + \omega\inv \cdot \omega = \omega\inv \cdot (\omega + \omega) = \omega\inv \cdot \omega = 1$. Thus, $r + r = r$ for all $r \in R$, and $R$ is a join semilattice with $r \vee s = r + s$; we define $r \leq s$ if $r + s = s$ \cite{CliffordPreston}*{Theorem~1.12}.

By distributivity, $R^\times$ is a partially ordered group. Furthermore, it has a maximum element $m : = \sum_{r \in R} r$. We now calculate, for all $r \in R^\times$, that $r = r \cdot 1 = r \cdot m \cdot m\inv \leq m \cdot m \inv = 1$. This implies that $R^\times$ is trivial because $1 = r \cdot r\inv \leq r \cdot 1 = r \leq 1$ for all $r \in R^\times$.
\end{proof}

For all objects $X$ and $Y$, let $0_{X,Y}$ be the unique morphism $X \to Y$ that factors through $0$.

\begin{proposition}\label{D}
The two endomorphisms of $I$ are $0 := 0_{I, I}$ and $1:= \id_{I}$, and $1+1 = 1$.
\end{proposition}

\begin{proof}
The endomorphism set $\cat{C}(I, I)$ is a complete semiring for the operations
$$
\sum_{\alpha \in M} r_\alpha := \Delta^\dag \circ \left(\bigoplus_{\alpha \in M} r_\alpha \right) \circ \Delta, \qquad \qquad r \cdot s := r \circ s;$$
see section~\ref{complete semirings}. The multiplicative identity $1$ is nonzero because $I$ is nonzero by assumption, and the nonzero elements of $\cat{C}(I, I)$ are invertible by assumption. Therefore, by Lemma~\ref{C}, the only nonzero element of $\cat{C}(I,I)$ is the identity $1$, and $1+1 = 1$.
\end{proof}

\begin{proposition}\label{E}
Let $X$ and $Y$ be objects of $\cat{C}$. We can partially order the morphisms $X \to Y$ by $r \leq s$ if $r + s = s$. Then, $\cat{C}(X, Y)$ is a complete lattice with $\bigvee_{\alpha \in M} r_\alpha = \sum_{\alpha \in M} r_\alpha$.
\end{proposition}

\begin{proof}
For all $a\: I \to X$ and all $r \: X \to Y$, we calculate that
\begin{align*}&
(r + r) \circ a = r \circ a + r \circ a = r \circ a \circ 1 + r \circ a \circ 1 = r \circ a \circ (1 + 1) = r \circ a \circ 1 = r \circ a.
\end{align*}
Since $I$ is a separator, we conclude that $r+r = r$ for all $r\: X \to Y$. Hence, $\cat{C}(X, Y)$ is an idempotent commutative monoid. Therefore, it is a poset with the given order, and moreover, $r_1 + r_2$ is the join of morphisms $r_1, r_2\: X \to Y$.

The same reasoning is sound for infinitely many summands. For all morphisms $a\: I \to X$ and $r\: X \to Y$ and all nonempty sets $M$, we calculate that
\begin{align*}&
\left(\sum_{\alpha \in M} r\right) \circ a = \sum_{\alpha \in M} r \circ a = \sum_{\alpha \in M} r \circ a \circ 1 = r \circ a \circ \sum_{\alpha \in M} 1 = r \circ a \circ 1 = r \circ a,
\end{align*}
where $\sum_{\alpha \in M} 1 = 1$ because $\sum_{\alpha \in M} 1$ is clearly an upper bound for $1$ in $\cat{C}(I,I)$. Therefore, $\sum_{\alpha \in M} r = r$ for all morphisms $r\: X \to Y$ and all nonempty sets $M$.

Let $\{r_\alpha\}_{\alpha \in M}$ be any nonempty indexed family of morphisms $X \to Y$. The sum $\sum_{\alpha \in M} r_\alpha$ is clearly an upper bound. Let $s$ be another upper bound. Then, $r_\alpha + s = s$ for all $\alpha \in M$, and hence
$$
s = \sum_{\alpha \in M} s = \sum_{\alpha \in M} (r_\alpha + s) = \sum_{\alpha \in M} r_\alpha + \sum_{\alpha \in M} s = \left( \sum_{\alpha \in M} r_\alpha \right) + s.
$$
We conclude that $\sum_{\alpha \in M} r_\alpha \leq s$ and, more generally, that $\sum_{\alpha \in M} r_\alpha$ is the least upper bound of $\{r_\alpha\}_{\alpha \in M}$. Therefore, $\cat{C}(X, Y)$ is a complete lattice with $\bigvee_{\alpha \in M} r_\alpha = \sum_{\alpha \in M} r_\alpha$.
\end{proof}

\begin{definition}\label{F}
For each object $X$, let $\top_X$ be the maximum morphism $I \to X$, i.e., let
$$
\top_X  = \sum_{a \: I \to X} a,
$$
and let $0_X$ be the minimum morphism $I \to X$, i.e., let $0_X = 0_{I,X}$.
\end{definition}

We will soon show that the $\ker(\top_X^\dag)$ is zero. To avoid clutter, we choose a representative for each isomorphism class of dagger kernels into $X$, so that for all morphisms $r$ and $s$ out of $X$, the kernels $\ker(r)$ and $\ker(s)$ are uniquely defined and furthermore $\ker(r) = \ker(s)$ whenever $\ker(r) \iso \ker(s)$. If the objects of $\cat{C}$ form a proper class, and if our foundations do not allow us to choose representative dagger kernels for each of them, then we make such choices only as necessary.   

\begin{proposition}\label{G}
Let $r\: X \to Y$. We have that  $r=0_{X,Y}$ if and only if $r \circ \top_X = 0_{Y}$. Furthermore, $\coker(r) = \coker(r \circ \top_X)$.
\end{proposition}

\begin{proof}
The forward direction of the equivalence is trivial. For the backward direction, assume that $r \circ \top_X = 0_Y$. By the monotonicity of composition in the second variable, we have that $r \circ a = 0_Y$ for all $a\: I \to X$. Because $I$ is a separator, we conclude that $r = 0$, as desired. Hence, we have proved the equivalence.

To prove the equality, we compare $\coker(r)\: X \to A$ and $\coker(r \circ \top_X)\: X \to B$. We first observe that $\coker(r) \circ r \circ \top_X = 0_A$, so $\coker(r)$ factors through $\coker(r \circ \top_X)$. Next, we observe that $\coker(r \circ \top_X) \circ r \circ \top_X = 0_B$. Via the proved equivalence, we infer that $\coker(r \circ \top_X) \circ r = 0_{X,B}$, so $\coker(r \circ \top_X)$ factors through $\coker(r)$. It follows that $\coker (r)$ and $\coker (r \circ \top_X)$ are equal.
\end{proof}

\begin{definition}\label{H}
For each morphism $a\: I \to X$, let $\Not a$ be the maximum morphism $I \to X$ such that $a^\dag \circ \Not a = 0$. 
\end{definition}

\begin{lemma}\label{I}
Let $a\: I \to X$. Then, $j = \ker(a^\dag)^\perp$ satisfies $a = j \circ \top_{A}$ and $\Not a = j^\perp \circ \top_{A^\perp}$, where $A$ is the domain of $j$ and $A^\perp$ is the domain of $j^\perp$.
\end{lemma}

\begin{proof}
For all $b\:I \to X$, we have the following chain of equivalences:
\begin{align*}
(j \circ \top_A)^\dag \circ b = 0 & \EV \top_A^\dag \circ j^\dag \circ b = 0 \EV j^\dag \circ b = 0 \EV (\exists c)\; b = \ker(j^\dag) \circ c
\\ & \EV (\exists c)\; b = j^\perp \circ c \EV (\exists c)\; b = \ker(a^\dag) \circ c \EV a^\dag \circ b = 0.
\end{align*}
The second equivalence follows by Proposition \ref{G}. The second-to-last equivalence follows by \cite{HeunenJacobs}*{Lemma 3}. Because $I$ is a separator, we conclude that $(j \circ \top_A)^\dag = a^\dag$ or equivalently that $j \circ \top_A = a$.

We prove the equation $\Not a = j^\perp \circ \top_{A^\perp}$ as a pair of inequalities. In one direction, we calculate that $a^\dag \circ j^\perp \circ \top_{A^\perp} = a^\dag \circ \ker(a^\dag) \circ \top_{A^\perp} = 0$, concluding that $j^\perp \circ \top_{A^\perp} \leq \Not a$. In the other direction, we reason that
\begin{align*}&
a^\dag \circ \Not a = 0 \IM (\exists c)\; \Not a = \ker(a^\dag) \circ c = j^\perp \circ c \IM \Not a \leq j^\perp \circ \top_{A^\perp}.
\end{align*}
Therefore, $\Not a = j^\perp \circ \top_{A^\perp}$, as claimed.
\end{proof}

\begin{proposition}\label{J}
For each object $X$, the lattice $\cat{C}(I, X)$ is an ortholattice when it is equipped with the orthocomplement $a \mapsto \Not a$. In other words, $\Not \Not a = a$, $a \wedge \Not a = 0_X$, $a \vee \Not a = \top_X$, and $a \leq b$ implies that $\Not b \leq \Not a$ for all $a, b\: I \to X$.
\end{proposition}

\begin{proof}
The operation $a \mapsto \Not a$ is antitone as an immediate consequence of Definition \ref{H}, and we now show that it is furthermore an order-reversing involution. Let $b = \Not a$. By Lemma~\ref{I}, the morphisms $j = \ker (a^\dag)^\perp\: A \to X$ and $k = \ker(b^\dag)^\perp\: B \to X$ are such that $a = j \circ \top_A$, that $\Not a = j^\perp \circ \top_{A^\perp}$, that $b = k \circ \top_B$, and that $\Not b = k^\perp \circ \top_{B^\perp}$. By Proposition~\ref{G}, $$k = \ker(b^\dag)^\perp = \ker(\top_{A^\perp}^\dag \circ j^{\perp\dag})^\perp = \ker(j^{\perp\dag})^\perp = j^{\perp\perp\perp} = j^\perp.$$ Thus, $k^\perp = j^{\perp \perp} = j$, and $\Not \Not a = \Not b = k^\perp \circ \top_{B^\perp} = j \circ \top_A = a$. Therefore, $a \mapsto \Not a$ is indeed an order-reversing involution. For all $a\: I \to X$, we also have that $(a \And \Not a)^\dag \circ (a \And \Not a) \leq a^\dag \circ \Not a = 0$ and thus that $a \And \Not a = 0_X$. Dually, $a \Or \Not a = \Not \Not a \Or \Not a = \Not (\Not a \And a) = \Not 0_X = \top_X$. Thus, $\Not a$ is a complement of $a$ for all $a \: I \to X$, and therefore, $\cat{C}(I, X)$ is an ortholattice.
\end{proof}

\begin{lemma}\label{K}
Let $j\: A \to X$ be a dagger kernel. Then, $j \circ j^\dag + j^\perp \circ j^{\perp\dag} = \id_X$.
\end{lemma}

\begin{proof}
Let $i = [j, j^\perp]\: A \oplus A^\perp \to X$, where the bracket notation refers to the universal property of the coproduct. Let $\inc_1\: A \to A \oplus A^\perp$ and $\inc_2\: A^\perp \to A \oplus A^\perp$ be the coproduct inclusions. We calculate that $\inc_1^\dag \circ i^\dag \circ i \circ \inc_1 = j^\dag \circ j = \id_A = \inc_1^\dag \circ \id_{A \oplus A^\perp} \circ \inc_1$, and similarly, $\inc_2^\dag \circ i^\dag \circ i \circ \inc_2 = \inc_2^\dag \circ \id_{A \oplus A^\perp} \circ \inc_2$. We also calculate that $\inc_1^\dag \circ i^\dag \circ i \circ \inc_2 = j^\dag \circ j^\perp = 0_{A^\perp, A} = \inc_1^\dag \circ \id_{A \oplus A^\perp} \circ \inc_2$, and dually, $\inc_2^\dag \circ i^\dag \circ i \circ \inc_1 = \inc_2^\dag \circ \id_{A \oplus A^\perp} \circ \inc_1$. We conclude that $i^\dag \circ i = \id_{A \oplus A^\perp}$, in other words, that $i$ is dagger monic. It is also epic because $j$ and $j^\perp$ are jointly epic by assumption. Therefore, $i$ is a dagger isomorphism. We now calculate that
$$
\id_X = i \circ i^\dag = [j,j^\perp] \circ [j, j^\perp]^\dag = \nabla_X \circ ( j \oplus j^\perp) \circ (j \oplus j^\perp)^\dag \circ \nabla^\dag_X = j \circ j^\dag + j^\perp \circ j^{\perp\dag}.
$$
\end{proof}

\begin{theorem}\label{L}
For each object $X$, the lattice $\cat{C}(I,X)$ is a complete Boolean algebra.
\end{theorem}

\begin{proof}
We have already shown that $\cat{C}(I, X)$ is a complete ortholattice. It remains to prove the distributive law. Let $a\: I \to X$. We will show that $b \mapsto a \And b$ distributes over joins.

Let $b\: I \to X$. By Lemma~\ref{I}, the dagger kernel $j = \ker(a^\dag)^\perp\: A \to X$ satisfies $j \circ \top_A = a$. We claim that $j \circ j^\dag \circ b = a \And b$. We certainly have that $j \circ j^\dag \circ b \leq j \circ \top_A = a$, and by Lemma~\ref{K}, we also have that $j \circ j^\dag \circ b \leq j \circ j^\dag \circ b + j^\perp \circ j^{\perp\dag}\circ b = b$. Thus, $j \circ j^\dag \circ b$ is a lower bound for $a$ and $b$.

Let $c\: I \to X$ be any lower bound for $a$ and $b$. Then, $(\Not a)^\dag \circ c \leq (\Not a)^\dag \circ a = 0$, so $c = \ker((\Not a)^\dag) \circ d$ for some morphism $d$. Applying Lemma~\ref{I} again, we calculate that
$
c = \ker(\top_{A^\perp}^\dag \circ j^{\perp\dag}) \circ d = \ker(j^{\perp\dag}) \circ d= j^{\perp\perp} \circ d = j \circ d.  
$
It follows that 
$$
c = j \circ d  = j \circ j^\dag \circ j \circ d = j\circ j^\dag \circ c \leq j \circ j^\dag \circ b.
$$
Therefore, $j \circ j^\dag \circ b = a \And b$ for all $b\: I \to X
$.

Let $b_1, b_2 \: I \to X$. We calculate that
$$
a \And (b_1 \Or b_2) = j \circ j^\dag \circ (b_1 + b_2) = j \circ j^\dag \circ b_1 + j \circ j^\dag \circ b_2 = (a \And b_1) \Or (a \And b_2).
$$
Therefore, $a \And (b_1 \Or b_2) = (a \And b_1) \Or (a \And b_2)$ for all $a, b_1, b_2\: I \to X$. We conclude that $\cat{C}(I, X)$ is a Boolean algebra.
\end{proof}

\section{Characterizations of $\cat{Rel}$}\label{part2}

Additionally, assume that $(\cat{C}, \tensor, I, \dag)$ is dagger compact closed \cites{Selinger, AbramskyCoecke}. This means that each object has a dagger dual. Explicitly, for each object $X$, there exists an object $X^*$ and a morphism $\eta_X \: I \to X^* \tensor X$ such that $\eta_X$ and $\eta_{X^*} := \beta_{X^*, X} \circ \eta_X$ together satisfy
$
(\id_X \tensor \eta_X^\dagger) \circ (\eta_{X^*} \tensor \id_X) = \id_X
$
and
$
(\id_{X^*} \tensor \eta_{X^*}^\dagger) \circ (\eta_X \tensor \id_{X^*}) = \id_{X^*}.
$
Here, $\beta$ is the braiding, and we have suppressed the associator and the unitors. See section~\ref{dagger categories} for more details. More commonly, the dagger dual of $X$ is equivalently defined in terms of two morphisms $\eta_X\: I \to X^* \tensor X$ and $\epsilon_X\: X \tensor X^* \to I $ that are then related by $\epsilon_X^\dag = \beta_{X^*, X} \circ \eta_X$.

In any dagger compact closed category, we have a bijection $\cat{C}(X \tensor Y, Z) \to \cat{C}(Y, X^* \tensor Z)$ that is defined by $r \mapsto (\id_{X^*} \tensor r) \circ (\eta_X \tensor \id_Y)$. We use this bijection to show that the monoidal unit is a monoidal separator.

\begin{proposition}\label{M}
$I$ is a monoidal separator.
\end{proposition}

\begin{proof}
Let $r_1, r_2\: X \tensor Y \to Z$, and assume that $r_1 \circ (a \tensor b) = r_2 \circ (a \tensor b)$ for all $a\: I \to X$ and $b\: I \to Y$. This equation is equivalent to $r_1 \circ (\id_X \tensor b) \circ a = r_2 \circ (\id_X \tensor b) \circ a$. It follows that $r_1 \circ (\id_X \tensor b) = r_2 \circ (\id_X \tensor b)$ for all $b\: I \to Y$, because $I$ is a separator. Applying the canonical bijection $\cat{C}(X, Z) \to \cat{C}(I, X^* \tensor Z)$, we find that $(\id_{X^*} \tensor (r_1 \circ (\id_X \tensor b))) \circ \eta_X = (\id_{X^*} \tensor (r_2 \circ (\id_X \tensor b))) \circ \eta_X$. Now we compute that
\begin{align*}
(\id_{X^*} \tensor r_1) \circ (\eta_X \tensor \id_Y) \circ b
& =
(\id_{X^*} \tensor (r_1 \circ (\id_X \tensor b))) \circ \eta_X
=
(\id_{X^*} \tensor (r_2 \circ (\id_X \tensor b))) \circ \eta_X
\\ & =
(\id_{X^*} \tensor r_2) \circ (\eta_X \tensor \id_Y) \circ b.
\end{align*}
It follows that $(\id_{X^*} \tensor r_1) \circ (\eta_X \tensor \id_Y) = (\id_{X^*} \tensor r_2) \circ (\eta_X \tensor \id_Y)$, because $I$ is a separator. Since the function $r \mapsto (\id_{X^*} \tensor r) \circ (\eta_X \tensor \id_Y)$ is a bijection $\cat{C}(X \tensor Y, Z) \to \cat{C}(Y, X^* \tensor Z)$, we conclude that $r_1 = r_2$. More generally, we conclude that  $I$ is a monoidal separator.
\end{proof}

Recall that an element $x$ of a Boolean algebra is said to be an \emph{atom} if $a \leq x$ implies that $a = x$ or $a = 0$, where $0$ is the minimum element of the Boolean algebra.

\begin{lemma}\label{N}
Let $X$ be an object. If $\top_X^\dag \circ \top_X = 1$, then $\cat{C}(I,X)$ contains an atom.
\end{lemma}

\begin{proof}
Assume that $\top_X^\dag \circ \top_X = 1$, and assume that $\cat{C}(I, X)$ contains no atoms. Let $s\: X \to X$ be the morphism $s=\sup\{ \Not c \circ c^\dag \suchthat c\:I \to X\}$. Let $a$ be a nonzero morphism $I \to X$. By assumption, $a$ is not an atom, so $a = a_1 \Or a_2$ for some disjoint nonzero $a_1, a_2\: I \to X$. Hence,
\begin{align*}
s \circ a \geq  
((\Not a_1 \circ a_1^\dag) \Or (\Not a_2 \circ a_2^\dag)) \circ a  & = (\Not a_1 \circ a_1^\dag \circ a) \Or (\Not a_2 \circ a_2^\dag \circ a) \\ &  = \Not a_1 \Or \Not a_2 = \Not (a_1 \And a_2) = \Not 0_X = \top_X.
\end{align*}
We conclude that $s \circ a = \top_X$ for all nonzero $a\: I \to X$, and of course, $s\circ 0_X = 0_X$. Because $I$ is separating, it follows that $s = \top_X \circ \top_X^\dag$. 

The monoidal category $(\cat{C}, \tensor, I)$ has a trace because it is compact closed. The trace of an endomorphism $r\: X \to X$ is defined by $\Tr(r) = \eta_X^\dag \circ (\id_{X^*} \tensor r) \circ \eta_X \in \cat{C}(I,I)$. For the standard properties of the trace, see \cite{HeunenVicary}*{section~3.4.5}. Furthermore, the enrichment of $\cat{C}$ over complete monoids \cite{Laird}*{Proposition~2.3} immediately implies that $\Tr\: \cat C(X, X) \to \cat C(I, I)$ is a homomorphism of complete monoids. We use these properties to calculate that
\begin{align*}
1 & = \Tr(1) = \Tr (\top_X^\dag \circ \top_X) = \Tr(\top_X \circ \top_X^\dag) = \Tr\left( \bigvee_{c\: I \to X} \Not c \circ c^\dag\right) \\ & = \bigvee_{c\: I \to X} \Tr(\Not c \circ c^\dag) = \bigvee_{c\: I \to X} \Tr(c^\dag \circ \Not c) = \bigvee_{c\: I \to X} 0 = 0.
\end{align*}
This conclusion contradicts Proposition~\ref{D}. Therefore, $\cat{C}(I, X)$ has at least one atom.
\end{proof}

Recall that a Boolean algebra is said to be \emph{atomic} if every nonzero element is greater than or equal to an atom. A complete Boolean algebra that is atomic is also \emph{atomistic}, which means that every element is the join of some set of atoms.

\begin{theorem}\label{O}
Let $X$ be an object. Then $\cat{C}(I, X)$ is a complete atomic Boolean algebra.
\end{theorem}

\begin{proof}
Assume that $\cat{C}(I, X)$ is not atomic. It follows that there exists a nonzero morphism $a\: I \to X$ such that there exist no atoms $x \leq a$. By Lemma~\ref{I}, there exists a dagger kernel $j\: A \to X$ such that $j \circ \top_A = a$ and hence $\top_A^\dag \circ \top_A = a^\dag \circ a = 1$. By Lemma~\ref{N}, $\cat{C}(I, A)$ contains an atom $z$.

We claim that $j \circ z$ is an atom of $\cat{C}(I,X)$. This morphism is certainly nonzero, because $j^\dag \circ j \circ z = z \neq 0_A$. Let $b \leq j \circ z$ be nonzero too. Then, $j^\perp \circ j^{\perp\dag} \circ b \leq j ^\perp \circ j^{\perp \dag} \circ j \circ z = 0_{X}$, so
$$
j \circ j^\dag \circ b = j \circ j^\dag \circ b + j^\perp \circ j^{\perp\dag} \circ b  = b
$$
by Lemma~\ref{K}. Thus, $j^\dag \circ b \neq 0_A$ because otherwise, we would have that $b = j \circ j^\dag \circ b = 0_X$. Furthermore, $j^\dag \circ b \leq j^\dag \circ j \circ z = z$. Because $z$ is an atom, we conclude that $j^\dag \circ b = z$ and hence that $b = j \circ j^\dag \circ b = j \circ z$. Therefore, $j \circ z$ is an atom.

Of course, $j \circ z \leq j \circ \top_A = a$, so there is a contradiction with our choice of $a$. We conclude that $\cat{C}(I, X)$ is atomic after all.
\end{proof}

\begin{definition}\label{P}
For each object $X$, define $E(X)$ to be the set of atoms of $\cat{C}(I,X)$. For each morphism $r\: X \to Y$, define $E(r) = \{(x,y) \in E(X) \times E(Y) \suchthat y^\dag \circ r \circ x = 1\}.$
\end{definition}

We now show that $E$ is an equivalence of dagger symmetric monoidal categories $\cat{C} \to \cat{Rel}$. We will often appeal to the following elementary proposition.

\begin{proposition}\label{new}
Let $X$ be an object, and let $x_1, x_2 \in E(X)$. Then, $x_1 = x_2$ iff $x_1^\dag \circ x_2 = 1$.
\end{proposition}

\begin{proof}
Because $\cat{C}(I,X)$ is a Boolean algebra, we have that
$$
(x_1 \And x_2) \Or ( \Not x_1 \And x_2) = (x_1 \Or \Not x_1) \And x_2 = x_2,
$$
$$
(x_1 \And x_2) \And (\Not x_1 \And x_2) = x_1 \And \Not x_1 \And x_2 = 0_X.
$$
Since $x_2$ is an atom and $x_1\And x_2, \Not x_1 \And x_2 \leq x_2$, we infer that $x_1 \And x_2 = x_2$ iff if $\Not x_1 \And x_2 \neq x_2$. We now reason that
\begin{align*}
x_1  = x_2 &\EV x_2 \leq x_1 \EV x_1 \And x_2 = x_2 \EV \Not x_1 \And x_2 \neq x_2 \\ & \EV x_2 \not \leq \Not x_1 \EV x_1^\dag \circ x_2 \neq 0 \EV x_1^\dag \circ x_2 = 1.
\end{align*}
The first equivalence holds by the definition of an atom, and the fifth equivalence holds by Definition~\ref{H}. Thus, the proposition is proved.
\end{proof}

\begin{lemma}\label{Q}
Let $X$ be an object. Then, $\displaystyle \id_X = \bigvee_{x \in E(X)} x \circ x^\dag.$
\end{lemma}

\begin{proof}
We apply Proposition~\ref{new} to calculate that for all $a\: I \to X$,
\begin{align*}
\left(\bigvee_{x \in E(X)} x \circ x^\dag\right) \circ a  & = \left(\bigvee_{x \in E(X)} x \circ x^\dag\right) \circ \left(\bigvee_{\begin{smallmatrix} y \in E(X) \\ y \leq a \end{smallmatrix}} y\right)
\\ & =
\bigvee_{x \in E(X)}\bigvee_{\begin{smallmatrix} y \in E(X) \\ y \leq a \end{smallmatrix}} x \circ x^\dagger \circ y
= \bigvee_{\begin{smallmatrix} x \in E(X) \\ x \leq a \end{smallmatrix}} x = a = \id_X \circ a.
\end{align*}
We conclude the claimed equality because $I$ is a separator.
\end{proof}

\begin{lemma}\label{R}
$E$ is a dagger functor $\cat{C} \to \cat{Rel}$. This means that $E$ is a functor such that $E(r^\dag) = E(r)^\dag$ for all morphisms $r$ of $\cat{C}$.
\end{lemma}

\begin{proof}
Let $X$ be an object of $\cat{C}$. 
\begin{align*}
E(\id_X) 
& = 
\{(x_1, x_2) \in E(X) \times E(X) \suchthat x_2^\dag \circ \id_X \circ x_1 = 1\}
\\& =
\{(x_1, x_2) \in E(X) \times E(X) \suchthat x_1 = x_2\} = \id_{E(X)}.
\end{align*}
Let $r\: X \to Y$ and $s\: Y \to Z$ be morphisms of $\cat{C}$. We apply Lemma~\ref{Q} to calculate that
\begin{align*}
E(s \circ r) & = \{(x,z) \in E(X) \times E(Z) \suchthat z^\dag \circ s \circ r \circ x= 1\}
\\ & \textstyle = \{(x,z) \in E(X) \times E(Z) \suchthat  z^\dag \circ s \circ (\bigvee_{y \in E(Y)} y \circ y^\dag )\circ r \circ x= 1\}
\\ & \textstyle = \{(x,z) \in E(X) \times E(Z) \suchthat \bigvee_{y \in E(Y)} z^\dag \circ s \circ y \circ y^\dag\circ r \circ x= 1\}
\\ & \textstyle= \{(x,z) \in E(X) \times E(Z) \suchthat \bigvee_{y \in E(Y)} (z^\dag \circ s \circ y) \And (y^\dag\circ r \circ x) = 1\}
\\ &  = \{(x,z) \in E(X) \times E(Z) \suchthat z^\dag \circ s \circ y = 1 \text{ and } y^\dag\circ r \circ x= 1\text{ for some }y \in E(Y)\}
\\ & = \{(x,z) \in E(X) \times E(Z) \suchthat (y,z) \in s \text{ and } (x,y) \in r \text{ for some }y \in E(Y)\}
\\ & =
E(s) \circ E(r).
\end{align*}
Thus, $E$ is a functor. Furthermore,
\begin{align*}
E(r^\dag)  &
= \{E(Y) \times E(X) \suchthat x^\dag \circ r^\dag \circ y = 1 \}
\\ & = 
\{ E(Y) \times E(X) \suchthat y^\dag \circ r \circ x = 1\}
\\ & =
\{E(Y) \times E(X) \suchthat (x,y) \in E(r)\} = E(r)^\dag.
\end{align*}
Therefore, $E$ is a dagger functor.
\end{proof}

\begin{proposition}\label{S}
$E$ is a dagger equivalence $\cat{C} \to \cat{Rel}$. This means that $E$ is a full and faithful dagger functor and every set is dagger isomorphic to $E(X)$ for some object $X$ of $\cat{C}$.
\end{proposition}

\begin{proof}
Let $r, s\: X \to Y$. Assume that $E(r) = E(s)$, i.e., that $y^\dag \circ r \circ x = y^\dag \circ s \circ x$ for all atoms $x\: I \to X$ and all atom $y\: I \to Y$. Since $\cat{C}(I, X)$ and $\cat{C}(I, Y)$ are complete atomic Boolean algebras by Theorem~\ref{O}, we find that $b^\dag \circ r \circ a = b^\dag \circ s \circ a$ for all morphisms $a\: I \to X$ and all morphisms $b\: I \to Y$. Appealing twice to our assumption that $I$ is a separator, we conclude that $r = s$. Therefore, $E$ is faithful.

Let $X$ and $Y$ be objects of $\cat{C}$, and let $R\: E(X) \to E(Y)$ be a binary relation. We reason that for all $x_0 \in E(X)$ and $y_0 \in E(Y)$,
\begin{align*}&
(x_0, y_0) \in  E\left(\bigvee_{(x,y) \in R} y \circ x^\dag\right)
\EV
y_0^\dag \circ \left(\bigvee_{(x,y) \in R} y \circ x^\dag\right) \circ x_0 = 1
\\ & \EV
\bigvee_{(x,y) \in R} y_0^\dag \circ y \circ x^\dag \circ x_0 = 1
\EV
\bigvee_{(x,y) \in R} (y_0^\dag \circ y) \And (x^\dag \circ x_0) = 1
\\ & \EV
y_0^\dag \circ y = 1 \text{ and } x^\dag \circ x_0 = 1 \text{ for some } (x,y) \in R
\EV
(x_0, y_0) \in R.
\end{align*}
We conclude that $E\left(\bigvee_{(x,y) \in R} y \circ x^\dag\right) = R$. Therefore, $E$ is full.

Let $M$ be a set. Let $X = \bigoplus_{m \in M} I$, and for each $m \in M$, let $j_m\: I \to X$ be the inclusion morphism for the summand of index $m$. We prove that $j_m$ is an atom. Let $a\: I \to X$ be a nonzero morphism such that $a \leq j_m$. It follows that $a^\dag \circ j_m \geq a^\dag \circ a = 1$. Furthermore, for all $m' \neq m$, we have that $a^\dag \circ j_{m'} \leq j_m^\dag \circ j_{m'} = 0$. By the universal property of $X$, we conclude that $a^\dag = j_m^\dag$ or equivalently that $a = j_m$. Therefore, $j_m$ is an atom for all $m \in M$.

Suppose that there is an atom $x\: I \to X$ such that $x \neq j_m$ for all $m \in M$. Then $x^\dag \circ j_m = 0$. By the universal property of $X$, we conclude that $x^\dag = 0_{X, I}$, contradicting that $x$ is an atom. Thus, $E(X) = \{j_m \suchthat m \in M\}$. The function $m \mapsto j_m$ is a dagger isomorphism $M \to E(X)$ in $\cat{Rel}$ because it is a bijection. Therefore, every set is dagger isomorphic to $E(X)$ for some object $X$ of $\cat{C}$.
\end{proof}

Finally, we prove that $E$ is a monoidal functor. We suppress unitors throughout.

\begin{lemma}\label{T}
Let $X$ and $Y$ be objects of $\cat C$. Then, $x \tensor y \in E(X \tensor Y)$ for all $x \in E(X)$ and $y \in E(Y)$, and this defines a bijection $\mu_{X,Y}\: E(X) \times E(Y) \to E(X \tensor Y)$.
\end{lemma}

\begin{proof}
Let $x \in E(X)$ and $y \in E(Y)$. Then, $x \tensor y$ is nonzero because $(x \tensor y)^\dag \circ (x \tensor y) = 1$. The Boolean algebra $\cat{C}(I, X \tensor Y)$ is atomic, so there is an atom $z\in E(X \tensor Y)$ such that $z \leq x \tensor y$. We now show that $z = x \tensor y$ by appealing to the fact that $I$ is a monoidal separator by Lemma~\ref{M}.

Let $a\: I \to X$ and $b\: I \to Y$. If $x \leq \Not a$ or $y \leq \Not b$, then $x^\dag \circ a = 0$ or $y^\dag \circ b = 0$, so
$$
z^\dag \circ (a \tensor b) \leq (x \tensor y)^\dag \circ (a \tensor b) = (x^\dag \circ a) \tensor (y^\dag \circ b) = 0
$$
and thus $z^\dag \circ (a \tensor b) = 0 = (x \tensor y)^\dag \circ (a \tensor b)$. If $x \leq a$ and $y \leq b$, then
$$
z^\dag \circ (a \tensor b) \geq z^\dag \circ (x \tensor y) \geq z^\dag \circ z = 1,
$$
and thus $z^\dag \circ (a \tensor b) = 1 = (x \tensor y)^\dag \circ (a \tensor b)$. Therefore, $z^\dag \circ (a \tensor b) = (x \tensor y)^\dag \circ (a \tensor b)$ for all $a\: I \to X$ and $b\: I \to Y$, and we conclude that $z^\dag = (x \tensor y)^\dag$ or equivalently that $z = x \tensor y$. Consequently, $x \tensor y$ is an atom.

We have shown that $x \tensor y \in E(X \tensor Y)$ for all $x \in E(X)$ and $y \in E(Y)$, and hence $(x, y) \mapsto (x \tensor y)$ defines a function $\mu_{X,Y}\: E(X) \times E(Y) \to E(X \tensor Y)$. This function is injective because $(x_1 \tensor y_1)^\dag \circ (x_2 \tensor y_2) = (x_1^\dag \circ x_2) \tensor (y_1^\dag \circ y_2) = 0$ whenever $x_1 \neq x_2$ or $y_1 \neq y_2$. This function is surjective because, by Lemma~\ref{Q}, for all $z \in E(X \tensor Y)$, we have that
$$
z = \id_{X \tensor Y} \circ z = (\id_X \tensor \id_Y) \circ z
=
\bigvee_{x \in E(X)} \bigvee_{y \in E(Y)} (x \tensor y) \circ (x \tensor y)^\dag \circ z
$$
and thus $(x \tensor y)^\dag \circ z \neq 0$ for some $(x,y) \in E(X) \times E(Y)$. Therefore, $\mu_{X, Y}$ is a bijection.
\end{proof}

\begin{proposition}\label{U}
$E$ is a strong symmetric monoidal functor $(\cat{C}, \tensor, I) \to (\cat{Rel}, \times, \{\ast\})$:
\begin{enumerate}
\item the isomorphism $\{\ast\} \to E(I)$ is the function $\ast \mapsto 1$;
\item the natural isomorphism $E(X) \times E(Y) \to E(X \tensor Y)$ is the function $(x, y) \mapsto x \tensor y$.
\end{enumerate}
\end{proposition}

\begin{proof}
For all objects $X$, $Y$ and $Z$, let $a_{X, Y, Z}\: (X \tensor Y) \tensor Z \to X \tensor (Y \tensor Z)$ be the associator in $\cat{C}$, and for all sets $L$, $M$, and $N$, let $\alpha_{L, M, N}\: (L \times M) \times N \to L \times (M \times N)$ be the associator in $\cat{Rel}$. We prove that the following diagram commutes:
$$
\begin{tikzcd}[column sep = 10em]
(E(X) \times E(Y)) \times E(Z)
\arrow{r}{\alpha_{E(X), E(Y), E(Z)}}
\arrow{d}[swap]{\mu_{X,Y} \times \id_Z}
&
E(X) \times (E(Y) \times E(Z))
\arrow{d}{\id_X \times \mu_{Y,Z}}
\\
E(X \tensor Y) \times E(Z)
\arrow{d}[swap]{\mu_{X \tensor Y, Z}}
&
E(X) \times E(Y \tensor Z)
\arrow{d}{\mu_{X, Y \tensor Z}}
\\
E((X \tensor Y) \tensor Z)
\arrow{r}[swap]{E(a_{X, Y, Z})}
&
E(X \tensor (Y \tensor Z))
\end{tikzcd}
$$
The six morphisms in this diagram are binary relations that are functions. In particular, $E(a_{X, Y, Z})$ consists of pairs $(((x_1 \tensor y_1) \tensor z_1), (x_2 \tensor (y_2 \tensor z_2)))$ that satisfy the following equivalent conditions:
\begin{align*}
(x_2 \tensor (y_2 \tensor z_2))^\dag \circ a_{X, Y, Z}\circ & ((x_1 \tensor y_1) \tensor z_1) = 1
\\ & \EV
(x_2 \tensor (y_2 \tensor z_2))^\dag \circ (x_1 \tensor (y_1 \tensor z_1)) = 1
\\ & \EV
((x_2^\dag \circ x_1) \tensor ((y_2^\dag \circ y_1) \tensor (z_2^\dag \circ z_1))= 1
\\ & \EV
x_1 = x_2 \; \text{and} \; y_1 = y_2 \; \text{and} \, z_1 = z_2.
\end{align*}
We can now prove that the diagram commutes via function application. We simply compute that for all $x \in E(X)$, $y \in E(Y)$, and $z \in E(Z)$, we have that
\begin{align*}&
(E(a_{X,Y,Z}) \circ \mu_{X \tensor Y, Z} \circ (\mu_{X,Y} \times \id_Z))((x, y), z)
=
(E(a_{X,Y,Z}) \circ \mu_{X \tensor Y, Z})(x \tensor y, z)
\\ & = 
E(a_{X,Y,Z})((x \tensor y) \tensor z)
=
x \tensor (y \tensor z)
=
\mu_{X, Y \tensor Z}(x, y \tensor z)
\\ & =
(\mu_{X, Y \tensor Z} \circ (\id_X \times \mu_{Y,Z}))(x, (y, z))
=
(\mu_{X, Y \tensor Z} \circ (\id_X \times \mu_{Y,Z}) \circ a_{E(X), E(Y), E(Z)})((x,y), z).
\end{align*}

We conclude that $E$ together with the natural bijection $\mu_{X, Y}\: E(X) \times E(Y) \to E(X \tensor Y)$ is a strong monoidal functor. The canonical bijection $\{\ast\} \to E(I)$ for this monoidal functor is evidently the unique such bijection \cite{EtingofGelakiNikshychOstrik}*{section 2.4}.

We verify that $E$ respects the braiding. For all objects $X$ and $Y$, let $b_{X,Y}\: X \tensor Y \to Y \tensor X$ be the braiding in $\cat{C}$, and for all sets $M$ and $N$, let $\beta_{M, N}\: M \times N \to N \times M$ be the braiding in $\cat{Rel}$. We prove that the following diagram commutes:
$$
\begin{tikzcd}[column sep = 6em]
E(X) \times E(Y)
\arrow{d}[swap]{\mu_{X,Y}}
\arrow{r}{\beta_{E(X), E(Y)}}
&
E(Y) \times E(X)
\arrow{d}{\mu_{Y,X}}
\\
E(X \tensor Y)
\arrow{r}[swap]{E(b_{X,Y})}
&
E(Y \tensor X)
\end{tikzcd}
$$
As before, the four morphisms in this diagram are binary relations that are functions. In particular, $E(b_{X,Y})$ consists of pairs $(x_1 \tensor y_1, y_2 \tensor x_2)$ that satisfy the following equivalent conditions:
\begin{align*}&
(y_2 \tensor x_2)^\dag \circ b_{X,Y} \circ (x_1 \tensor y_1) = 1
\EV
(y_2 \tensor x_2)^\dag \circ (y_1 \tensor x_1) = 1
\\ & \EV
(y_2^\dag \circ y_1) \tensor (x_2^\dag \circ x_1) = 1
\EV
x_1 = x_2 \text{ and } y_1 = y_2.
\end{align*}
We can now prove that the diagram commutes via function application. We simply compute that for all $x \in E(X)$ and $y \in E(Y)$, we have that
\begin{align*}
(E(b_{X,Y}) \circ \mu_{X,Y})(x,y)
=
E(b_{X,Y})(x \tensor y) 
=
y \tensor x
=
\mu_{Y, X}(y, x)
=
(\mu_{Y,X} \circ \beta_{E(X), E(Y)})(x, y).
\end{align*}
Therefore, $E$ is a strong symmetric monoidal functor.
\end{proof}

\begin{theorem}\label{V}
Let $(\cat{C}, \tensor, I, \dag)$ be a dagger compact closed category. If
\begin{enumerate}
\item each family of objects has a dagger biproduct,
\item each morphism has a kernel that is dagger monic,
\item $k$ and $k^\perp$ are jointly epic for each dagger kernel $k$,
\item $I$ is nonzero,
\item each nonzero morphism $I \to I$ is invertible,
\item $I$ is a separator,
\end{enumerate}
then the functor $E\: \cat{C} \to \cat{Rel}$ of Definition~\ref{P} is a strong symmetric monoidal dagger equivalence. Conversely, it is routine to verify that $(\cat{Rel}, \times, \{\ast\}, \dag)$ is a dagger compact closed category satisfying (1)--(6).
\end{theorem}

\begin{proof}
Combine Propositions \ref{S} and \ref{U}.
\end{proof}

Assuming sufficient choice, the adjoint of $E$ \cite{MacLane2}*{Theorem IV.4.1} can be selected to be a dagger functor \cite{Vicary}*{Lemma 5.1} and can then be made a strong symmetric monoidal functor \cite{EtingofGelakiNikshychOstrik}*{Remark 2.4.10}. As corollary of Theorem~\ref{V}, we obtain a characterization of $\cat{Rel}$ that is more in the spirit of mathematical logic.

\begin{corollary}\label{W}
Let $(\cat{C}, \tensor, I, \dag)$ be a dagger compact closed category. If
\begin{enumerate}[\quad\;\,(1')]
\item each family of objects has a dagger biproduct,
\item $I$ is simple and separating,
\item each object $X$ has a unique morphism $\top_X\: I \to X$ such that $\coker(\top_X) = 0$,
\item each morphism $a\: I \to X$ has a dagger isomorphism $i\: A \oplus B \to X$ such that
$$
\begin{tikzcd}
I
\arrow{r}{a}
\arrow{d}[swap]{\top_A}
&
X
\\
A
\arrow{r}[swap]{\inc_1}
&
A \oplus B,
\arrow{u}[swap]{i}
\end{tikzcd}
$$
\end{enumerate}
then the functor $E\: \cat{C} \to \cat{Rel}$ of Definition~\ref{P} is a strong symmetric monoidal dagger equivalence. Conversely, it is routine to verify that $(\cat{Rel}, \times, \{\ast\}, \dag)$ is a dagger compact closed category satisfying (1')--(4').
\end{corollary}

We may gloss these conditions as expressing that (1') disjoint unions of sets exist, (2') the monoidal unit is a singleton set, (3') every set has a coempty predicate, and (4') every predicate on a set determines a subset of that set. In particular, condition (4') recalls the axiom of separation in set theory.

\begin{proof}[Proof of Corollary~\ref{W}]
Assume (1')--(4'). First, we claim that for all $a\: I \to X$, if $a^\dag \circ a = 0$, then $a = 0$. Applying assumption (4'), we write $a = i \circ \inc_1 \circ \top_A$, where $i$ is a dagger isomorphism and $\coker(\top_A) = 0$. Assume $a^\dag \circ a = 0$. Then, $0 = a^\dag \circ a = \top_A^\dag \circ {\inc_1}^\dag \circ i^\dag \circ i \circ \inc_1 \circ \top_A = \top_A^\dag \circ \top_A$. It follows that $\top_A^\dag$ factors through $0$. Thus, $\top_A$ and hence $a$ factor through $0$. We have established our first  claim.

Second, we claim that there are exactly two morphisms $I \to I$, namely, $0:= 0_I \neq \id_I$ and $1 := \top_I = \id_I$. Let $a\: I \to I$. By assumption (2'), $\coker(a) = \;!\: I \to 0$ or $\coker(a) = \id_I\: I \to I$ up to isomorphism. In the former case, $a = \top_I$ by assumption (3'), and in the latter case, $a = 0_I$. In particular, $\id_I = \top_I$ or $\id_I = 0_I$. In the latter case, $I \iso 0$, contradicting assumption (2'). Therefore, $\id_I \neq 0_I$, and hence $\id_I = \top_I$. We have established our second claim.

Thus, $(\cat{C}, \tensor, I, \dag)$ is a dagger compact closed category that satisfies assumptions (1), (4), (5), and (6) of Theorem~\ref{V}. It remains to show that $(\cat{C}, \tensor, I, \dag)$ satisfies assumptions (2) and (3) of Theorem~\ref{V}.

Let $r\: X \to Y$. Let $a =  r \circ \top_X$. By assumption (4'), there exists a dagger isomorphism $i\: A \oplus B \to Y$ such that $a = i \circ \inc_1 \circ \top_A$. We claim that $\inc_2^\dag \circ i^\dag$ is a cokernel of $r$. First, we  calculate that $\inc_2^\dag \circ i^\dag \circ r \circ \top_X = \inc^\dag \circ i^\dag \circ a = \inc_2^\dag \circ i^\dag \circ i \circ \inc_1 \circ \top_A = \inc_2^\dag \circ \inc_1 \circ \top_A = 0$. By assumption (3'), we have that $\inc_2^\dag \circ i^\dag \circ r = 0$.

Let $s\: Y \to Z$ be such that $s \circ r = 0_{X, Z}$. It follows that $s \circ i \circ \inc_1 \circ \top_A = s \circ a = s \circ r \circ \top_X = 0$. By assumption (3'), we have that $s \circ i \circ \inc_1 = 0$. As for any dagger biproduct of two objects, we have that $\coker(\inc_1)  = \inc_2^\dag$, and thus, $s \circ i = t \circ \inc_2^\dag$ for some morphism $t$. We conclude that $s = s \circ i \circ i^\dag = t \circ \inc_2^\dag \circ i^\dag$.

Therefore, $ \inc_2^\dag \circ i^\dag$ is a cokernel of $r$, as claimed. In other words $i \circ \inc_2$ is a kernel of $r^\dag$. The kernel $i \circ \inc_2$ is dagger monic, and hence we have verified assumption (2) of Theorem~\ref{V}. Furthermore, as for any dagger biproduct of two objects, we have that $\inc_1$ and $\inc_2$ are jointly epic and that $\inc_2^\perp = \inc_1$. Hence, $i \circ \inc_1$ and $i \circ \inc_2$ are jointly epic and $(i \circ \inc_2)^\perp = i \circ \inc_1$. We conclude that every dagger kernel is jointly epic with its orthogonal complement, verifying assumption (3) of Theorem~\ref{V}.

We have verified the assumptions of Theorem~\ref{V}, and we now apply it to obtain the desired conclusion.
\end{proof}

We now complete the proof of Theorem~\ref{A}, which provides comparable characterizations of the dagger symmetric monoidal categories $(\cat{Rel}, \times, \{\ast\}, \dag)$ and $(\cat{Hilb}_\FF, \tensor, \FF^1, \dag)$ for $\FF = \RR, \CC$.

\begin{proof}[Proof of Theorem~\ref{A}]
Axiom~D is just the existence of binary dagger biproducts \cites{Selinger, AbramskyCoecke2}. Indeed, any binary coproduct whose inclusions are orthogonal dagger kernels is clearly a dagger biproduct. Conversely, the inclusions of a binary dagger biproduct are orthogonal dagger kernels \cite{HeunenVicary}*{exercise 2.6}. By the same argument, the condition that every family of objects has a coproduct whose inclusions are pairwise-orthogonal dagger kernels is just the existence of all dagger biproducts. The backward implication of statement (i) is thus a corollary of Theorem~\ref{V}; the proof of the forward implication is routine.

Statement (ii) is a corollary of \cite{HeunenKornell}*{Theorem 10}: Assume axioms A--G, that every dagger monomorphism is a dagger kernel, and that the wide subcategory of dagger kernels has directed colimits. Then, $I$ is simple as a consequence of axioms E and F. Furthermore, there is a morphism $z\: I \to I$ such that $1 + z =0$. Indeed, suppose that there is no such morphism $z$, and let $\Delta_4\: I \to I \oplus I \oplus I \oplus I$ be the diagonal map. Then, $\Delta_4/2$ is a dagger monomorphism and hence a dagger kernel. Its cokernel is zero because $I$ is a separator and $(\Delta_4/2)^\dag \circ v \neq 0$ for all nonzero $v\: I \to I \oplus I \oplus I \oplus I$. Thus, $\Delta_4/2$ is an isomorphism \cite{HeunenJacobs}*{Lemma~2.3(iv)}, which contradicts the assumption that $I$ is nonzero. We conclude that $1$ has an additive inverse in $\cat{C}(I, I)$. It follows that each parallel pair of morphisms, $f$ and $g$, has a dagger equalizer, which is equal to the dagger kernel of $f-g$. Therefore, by \cite{HeunenKornell}*{Theorem 10}, $(\cat{C}, \tensor, I, \dag)$ is equivalent to $(\cat{Hilb_\FF}, \tensor, \FF^{1}, \dag)$ for $\FF = \RR$ or $\FF = \CC$. We have proved the backward implication of statement (ii); the proof of the forward implication is routine.
\end{proof}

\begin{remark}
It is routine to verify that the dagger compact closed category $(\cat{Rel}, \times, \{\ast\}, \dag)$ also has the property that the wide subcategory of dagger kernels has directed colimits. Indeed, the latter category is simply the category of sets and injections.
\end{remark}

\section*{Acknowledgements}

I thank 
John Baez,
Chris Heunen,
Martti Karvonen,
Bert Lindenhovius, and
Morgan Rogers
for their comments.
I am grateful to Bert Lindenhovius for noticing an issue in the proof of Proposition~\ref{E} in the original preprint.

\end{document}